\documentclass[11pt]{amsart}
\usepackage{pb-diagram,amssymb,color,epic,eepic,verbatim,epsfig}
\usepackage{graphicx}
\input xy
\xyoption{all}

\numberwithin{equation}{section}
\numberwithin{figure}{section}

\newtheorem{theorem}{Theorem}[section]
\newtheorem{corollary}[theorem]{Corollary}

\newtheorem*{thmA}{Theorem \ref{thmA}}

\newtheorem{proposition}[theorem]{Proposition}
\newtheorem{lemma}[theorem]{Lemma}
\newtheorem{lem}[theorem]{}
\theoremstyle{definition}
\newtheorem{definition}[theorem]{Definition}
\theoremstyle{remark}
\newtheorem{remark}[theorem]{Remark}
\newtheorem{example}[theorem]{Example}
\newcommand{\blem}{\begin{lem} \rm}
\newcommand{\elem}{\end{lem}}
\newcommand\A{\mathcal{A}}

\newcommand{\F}{\mathcal{F}}

\newcommand{\R}{\mathbb
{R}}
\newcommand{\br}{\mathbf{r}}
\newcommand{\bz}{\mathbf{z}}
\newcommand{\supp}{\on{supp}}

\newcommand{\C}{\mathbb{C}}

\newcommand{\Z}{\mathbb{Z}}

\newcommand{\on}{\operatorname}

\newcommand\bra[1]{ < \kern-.7ex {#1} \kern-.7ex >} 

\newcommand{\ssm}{\kern-.5ex \smallsetminus \kern-.5ex}

\newcommand\dirac{/\kern-1.2ex\partial} 
\newcommand\qu{/\kern-.7ex/} 
\newcommand\lqu{\backslash \kern-.7ex \backslash} 

\newcommand\dr{r_+ \kern-.7ex - \kern-.7ex r_-}
 
\newcommand{\labell}\label

\newcommand{\ol}{\overline}

\newcommand\bdefn{\begin{definition}}
\newcommand\edefn{\end{definition}}
\newcommand\bea{\begin{eqnarray*}}
\newcommand\eea{\end{eqnarray*}}
\newcommand\bcv{\left[ \begin{array}{r} }
\newcommand\ecv{\end{array} \right] }

\newcommand\bma{\left[ \begin{array} }
\newcommand\ema{\end{array} \right]}
\newcommand\ben{\begin{enumerate}}
\newcommand\een{\end{enumerate}}
\newcommand\bex{\begin{example}}
\newcommand\bsj{\left\{ \begin{array}{rrr} }
\newcommand\esj{\end{array} \right\}}

\newcommand\eex{\end{example}}

\newcommand\Ob{{\on{Ob}\ }}

\newcommand\sx{*\kern-.5ex_X}
\newcommand{\ainfty}{{$A_\infty$\ }}

\def\mathunderaccent#1{\let\theaccent#1\mathpalette\putaccentunder}
\def\putaccentunder#1#2{\oalign{$#1#2$\crcr\hidewidth \vbox
to.2ex{\hbox{$#1\theaccent{}$}\vss}\hidewidth}}

\newcommand{\bfs}{\mathbf{s}}

\newcommand{\bx}{\mathbf{x}}

\newcommand{\bfa}{\mathbf{a}}
\newcommand{\bfe}{\mathbf{e}}

\newcommand{\bw}{\mathbf{w}}

\newcommand{\B}{\mathcal{B}}

\begin{document}

\title{Quilted strips, graph associahedra, and \ainfty $n$-modules}
\author{Sikimeti Ma'u}
\address{Mathematical Sciences Research Institute,\\ 17 Gauss Way,\\ Berkeley, CA 94720-5070}
\email{sikimeti@msri.org}
\keywords{graph associahedra, associahedra, permutahedra, A-infinity, modules, bimodules, toric varieties, moment map}
\maketitle

\begin{abstract} We consider moduli spaces of {quilted strips} with markings and their compactifications.   Using the theory of moment maps of toric varieties we identify the compactified moduli spaces with certain graph associahedra.  We demonstrate how these moduli spaces govern the combinatorics of \ainfty $n$-modules, which are natural generalizations of \ainfty modules ($n=1$) and bimodules ($n=2$).
\end{abstract}  

\section{Introduction}

The combinatorics of compactified moduli spaces of Riemann surfaces are often responsible for rich algebraic structures in pseudoholomorphic curve theory.  The prototypical example is Fukaya's \ainfty category in symplectic topology, where the \ainfty algebraic relations come from the fact that moduli spaces of disks with markings realize the Stasheff polytopes (or {\em associahedra}), whose combinatorics govern the \ainfty associativity relations.   

In this note we describe the structure of moduli spaces of {\em quilted lines} with markings, whose natural Grothendieck-Knudsen compactifications realize certain graph associahedra.  The graph associahedra are generalizations\footnote{Here the reader should perhaps be warned that they are not the same as the {\em generalized associahedra} of Fomin and Zelevinsky.} of associahedra introduced by Carr and Devadoss in \cite{cardev}.  The main purpose of this note is to prove:    
\begin{theorem}\label{thmA}
The compactified moduli space of stable $n$-quilted lines with markings is a smooth manifold-with-corners, homeomorphic to the graph associahedron of its dual graph. 
\end{theorem}

The codimension one facets of the compactified moduli spaces govern the combinatorics of \ainfty-algebraic structures that we call $n$-modules (see Section 4).   They are direct generalizations of \ainfty modules ($n=1$) and bimodules ($n=2$), but do not seem to have appeared in the literature as yet.  Both the \ainfty module and bimodule equations are essentially reformulations of the \ainfty associativity relations; from the point of view of the moduli spaces of quilted lines, this reflects the fact that the moduli spaces of quilted lines for $n =1$ and $n=2$ all yield classical associahedra.  However for $n \geq 3$, the moduli spaces yield more general graph associahedra.  

Quilted lines belong to a class of generalized Riemann surfaces called {\em quilts}, which are essentially  Riemann surfaces decorated with certain real-analytic submanifolds, called {\em seams}.   Quilts are domains for a generalized pseudoholomorphic curve theory developed by Wehrheim and Woodward \cite{ww2,ww3}, in which the seams of a quilt are given boundary conditions in Lagrangian correspondences.  Applications of these moduli spaces of quilted lines to Lagrangian Floer theory will be treated in a separate paper.  It is also worth mentioning that these results are closely related to those of \cite{multiplihedra}, which dealt with moduli spaces of quilted disks (which are connected to \ainfty functors), and the method of proof is very similar.   

Interestingly, the graph associahedra appearing in this note also appear in the work of Bloom on a spectral sequence for links in monopole Floer homology \cite{bloom}.  These polytopes should be thought of as parametrizing a mixture of homotopy associativity and homotopy commutativity: they interpolate between associahedra, which parametrize homotopy associativity, and permutahedra, which parametrize homotopy commutativity.  

The material is organized as follows.  Section 2 defines moduli spaces of quilted lines with markings, and their compactifications.  Section 3 proves the correspondence between the compactified moduli spaces and certain graph associahedra.  Section 4 introduces the connection between the codimension one facets of these polytopes and the algebraic structure of \ainfty $n$-modules, describing some of the properties of $n$-modules which are analogous to well-known properties of \ainfty modules and bimodules.  \\
\\
\noindent {\it Acknowledgements.} This work was carried out while the author was a postdoctoral fellow at MSRI, during the 2009-2010 Program in Symplectic and Contact Geometry and Topology.  The author thanks MSRI and the program organizers for the stimulating and hospitable environment.  Thanks also to Satyan Devadoss for useful discussions about graph associahedra. 

\section{Moduli spaces of marked quilted lines}

Let $\C$ denote the complex plane, with complex coordinate $z = x+iy$.

\begin{definition}  Fix $n \geq 1$, and fix $-\infty < x_1 < x_2 < \ldots < x_n < \infty$.   An {\em $n$-quilted line} $Q$ consists of the $n$ parallel lines $l_1,\ldots,l_n$, where $l_j$ is the vertical line $\{ x_j + i \R \}$, considered as a subset of $\C$.   Let $\br = (r_1, \ldots, r_n) \in \Z^{n+1}_{\geq 0}$.  An {\em $n$-quilted line with $\br$ markings} consists of data $(Q, \bz_0, \ldots, \bz_n)$ where $Q_n$ is a quilted line, and each vector $\bz_j = (z_j^1, z_j^2, \ldots, z_{j}^{r_j})$ is an upwardly ordered configuration of points in $l_j$, i.e., $z_{j}^{k} = x_j + i y_{j}^{k}$ with $y_{j}^{0} < y_{j}^1 < \ldots < y_{j}^{r_j}$.    There is a free and proper $\R$ action on the space given by simultaneous translation in the $y$ direction.  The {\em moduli space of $n$-quilted, $\br$-marked lines}, $Q(n,\br)$, is the set of such marked quilted lines, modulo translation.     
\end{definition}

\subsection{Alternative picture} When $n\geq 2$, the whole strip can be holomorphically mapped to a disk with two distinguished boundary markings representing the $+\infty$ and $-\infty$ ends of the strip, and arcs connecting the two markings.  The distances between the lines are reflected in the angles at which the arcs meet.  (When $n=1$, one may add a second line with no markings on it to get a strip which can then be identified with a disk.) Thus, $Q(n, \br)$ is also the moduli space of marked disks with two distinguished points, and arcs of fixed angles connecting the two points carrying markings on the arcs, modulo complex automorphisms of the disk.  Very similarly, $Q(n,\br)$ can be realized as a moduli space of $n$ rays starting at $0$ in the upper half-plane, with marked points on the rays, modulo dilations; this realization of the moduli space is the one we will use the most in this note.  Three views of a marked quilted line are illustrated in Figure \ref{qstripfig}.

\begin{figure}[h]
\includegraphics[width=5in]{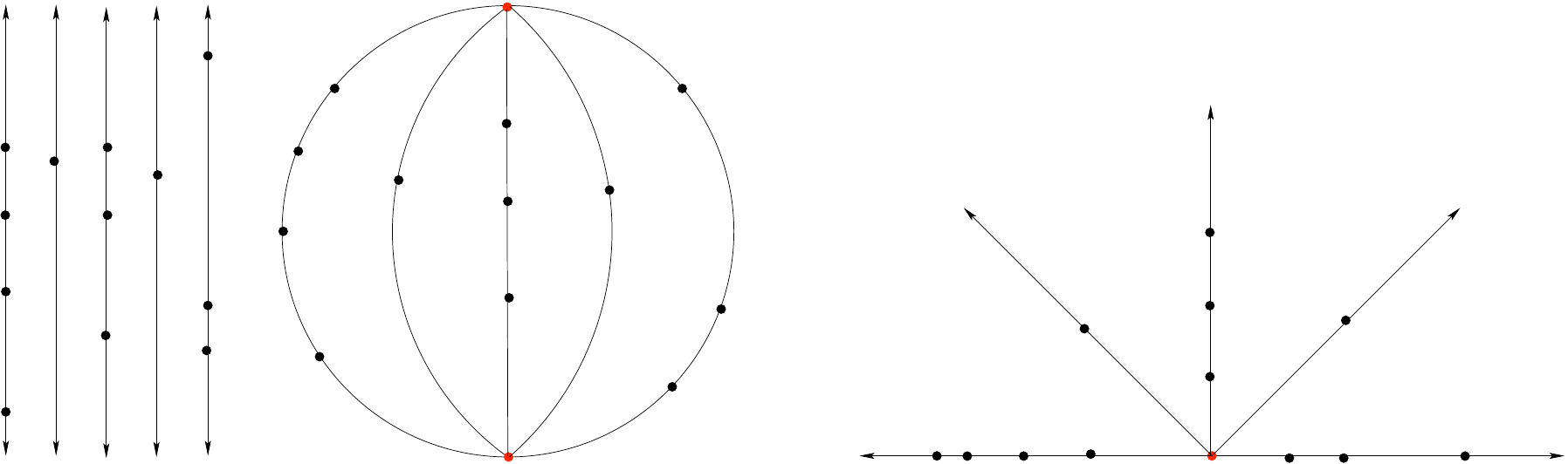}
\caption{An element of $Q(5,(4,1,3,1,3))$ viewed as {\it (left)} a quilted line, modulo translations; {\it (center)} a disk with arcs, modulo complex automorphisms; {\it (right)} rays from an origin, modulo dilations.}\label{qstripfig}
\end{figure}

\subsection{Bubble trees} A marked $n$-quilted line is {\em stable} provided that it has at least one marking.  The natural Grothendieck-Knudsen type compactification $\ol{Q}(n,\br)$ is obtained by allowing the vertical distances between markings to go to $\infty$, and markings on the same line to bubble.   More precisely,
we compactify with {\em stable, nodal $n$-quilted lines}, which are the relevant ``bubble trees'', whose  combinatorial types have the following description. 
\begin{definition}
Let $n \geq 1$, $\br = (r_1, \ldots, r_n) \in \Z_{\geq 0}^n$.  An {\em $n$-paged ribbon tree with $\br $  leaves} consists of:
\begin{itemize}
\item A set of vertices $V = V_{spine} \sqcup V_1 \sqcup \ldots \sqcup V_n$ with $|V_{spine}| \geq 1$;  

\item A set of finite edges $E = E_{spine} \sqcup E_1 \sqcup \ldots \sqcup E_n$;

\item A set of semi-infinite edges, $E^{\infty}_{spine} \cup E^{\infty}_1 \cup \ldots E^{\infty}_n$ where $E^{\infty}_{spine}= \{e_{+\infty}, e_{-\infty}\}$, and $E^{\infty}_{i} = \{ e_{i, 1}, \ldots, e_{i,r_i}\}$ for $i = 1$ to $n$.   

\item A ribbon structure\footnote{A ribbon structure is a cyclic ordering of the edges at each vertex. A ribbon structure on a tree determines, up to planar isotopy, a planar embedding of the tree, and therefore a cyclic ordering of the semi-infinite edges.}  for each subtree $T_i = (V_{spine}\cup_i V_i \cup_i V_i^\infty, E_{spine}\cup_i E_i \cup_i E_i^\infty )$, for $i = 1$ to $n$.
\end{itemize}
 
The data are required to satisfy the following conditions.  
 
\begin{enumerate}
\item The subgraph $(V_{spine}, E_{spine} \cup \{e_{+\infty}, e_{-\infty}\})$ is a path with $e_{+\infty}$ at one end and $e_{-\infty}$ at the other; we call this path the {\em spine}.  
 
\item For each $i$, the planar embedding of the subtree $T_i = (V_{spine}\cup V_i, E_{spine}\cup E_i \cup \{e_{+\infty}, e_{-\infty}\} \cup \{e_{i_1}, \ldots, e_{i,r_i}\}$ determined by the ribbon structure induces the cyclic order $e_{+\infty}, e_{i,1}, \ldots, e_{i, r_i}, e_{-\infty}$ on the semi-infinite edges.  
 
\item When $i \neq j$, there are no edges between $V_i$ and $V_j$. 
 
\item  Each vertex in the tree is adjacent to at least 3 edges (the {\em stability} condition).
 \end{enumerate}

The planar subtree $T_i = (V_{spine}\cup V_i,  E_{spine} \cup E_i \cup E_i^{\infty})$ is called the $i$-th {\em page} of the tree.  
\end{definition}

\begin{example}
See Figure \ref{egP}.
\end{example}

\begin{figure}[h]
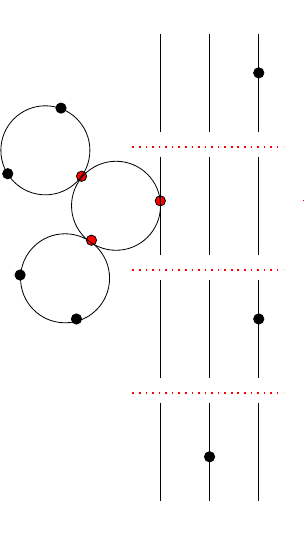\ \ \ \ \ \ \ \ \ \ \includegraphics[height=2in]{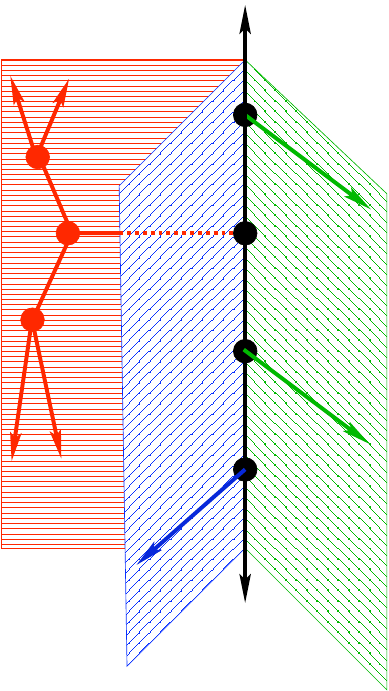}
\caption{A stable, nodal 3-quilted line in $\ol{Q}(3,(4,1,2))$ with its combinatorial 
type, a 3-paged ribbon tree.}\label{egP}
\end{figure}

\begin{definition}
Let $T$ be an $n$-paged ribbon tree with $\br = (r_1,\ldots,r_n)$ leaves.  A {\em stable, nodal $n$-quilted line with $\br $ markings} of combinatorial type $T$ is a tuple $( \{Q_v\}_{v\in V}, \{z_{v,e} \}_{v\in V, e \in E, v \in e}, \{z_{i,j}\}_{i = 1,\ldots,n, j=1,\ldots,r_i})$, where $Q_v$ are components of the bubble tree, $z_{v,e} \in Q_v$ are nodal points, and $z_{i,j}$ are markings.  
If $v \in V_{spine}$, the component $Q_v$ is an $n$-quilted line, otherwise, $v$ belongs to some page, say the $i$-th page, and $Q_v$ is a disk component whose boundary is connected to the $i$-th line.  For each finite edge $e = (v,w) \in E$, $z_{v,e} \in Q_v$ is a nodal point on the component $Q_v$; if $v$ is on the $i$-th page, the nodal point $z_{v,e}$ is connected to the $i$-th line. Each $z_{i,j}$ is a marked point on the component indexed by the vertex in $T$ that is adjacent to the leaf $e_{i,j}$, and lies on the $i$-th line.  
\end{definition}

\subsection{Equivalence} Two stable, nodal $n$-quilted lines with $\br$ markings of the same combinatorial type $T$ are {\em isomorphic} if there is a tuple $(\phi_v)_{v\in V}$ of reparametrizations $\phi_v: Q_v \to Q^\prime_v$ such that $\phi_v(z_{v,e}) = z^\prime_{v,e}$ for all $E \in E(T)$, and $\phi_v(z_{i,j}) = z^\prime_{i,j}$ for all $i = 1,\ldots,n$ and $j=1,\ldots,r_i$.  Write $Q_T$ for the set of isomorphism classes of nodal quilted lines with combinatorial type $T$.  

\subsection{Compactification} As a set, the {\em compactified moduli space of $n$-quilted lines with $\br$ markings } is the union  
\[
\ol{Q}(n,\br) := \bigcup_{T} Q_T
\]
where $T$ ranges over all $n$-paged trees with $\br$ leaves.

\subsection{Charts}
In this section we will define local charts on $\ol{Q}(n,\br)$ using coordinates that are closely related to cross-ratios.  First we fix some notation.  Identify the $n$-quilted line with a collection of $n$ rays from the origin, with marked points on the rays, and with the reparametrization group acting by dilations. 

\begin{definition} \label{coord}
The $r_i$ markings on the $i$-th ray are parametrized by homogeneous coordinates $(x_{i,1}:x_{i,2}:\ldots:x_{i,r_i})$ where $x_{i,1}$ is the distance between the first marking and $0$, and $x_{i,k}$ is the distance along the ray between the $k$-th marking and the $(k-1)$-st marking (see Figure \ref{homcoord}).  
\end{definition}

\begin{figure}[h]
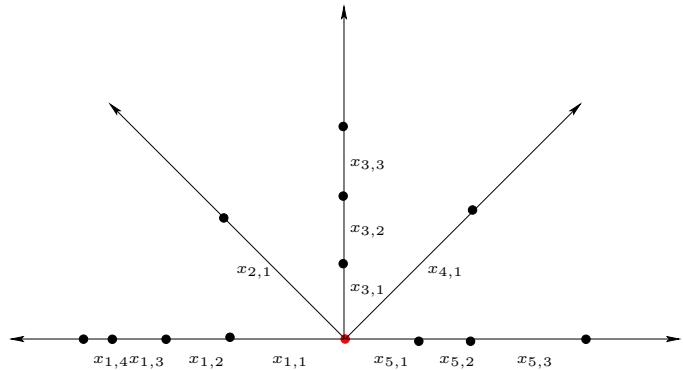
\caption{Homogeneous coordinates for $Q(5,(4,1,3,1,3))$.}\label{homcoord}
\end{figure}

Consider a maximal combinatorial type $T$, i.e., a zero-dimensional facet of $\ol{Q}(n,\br)$.  We associate a chart $\phi_T: [0,\infty)^{|E|} \to \ol{Q}(n,\br)$ to $T$ by assigning a ratio to each node (indexed by finite edges in $E$) as follows.  On each component $Q_v$ there is a unique special point which is closest to the marked point $+\infty$, and once that point is mapped to $+\infty$ there is a {\em unique} variable $x_{i,j}$ which is finite and non-zero on $Q_v$.  Consider a node between components $Q^+$ and $Q^-$, where $Q^+$ is the component closer to the distinguished marking $+\infty$.  Let $x_{i,j}$ be the variable on $Q^+$ that is finite and non-zero, and let $x_{k,l}$ be the variable on $Q^-$ that is finite and non-zero.  Assign the ratio $x_{k,l}/x_{i,j} \in [0,\infty) $ to that node, where the node corresponds to the ratio being $0$.  

\begin{example} Figure \ref{fig:egvert} illustrates the ratio chart for a maximal combinatorial type. 
\end{example}

\begin{figure}
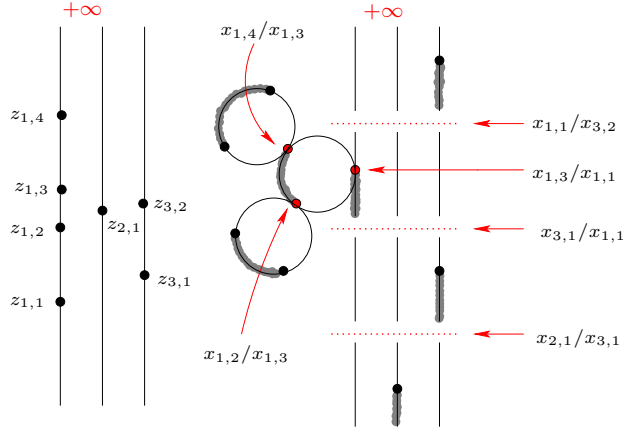
\caption{The ratio chart for a neighborhood of a vertex in $\ol{Q}(3,(4,1,2))$.  Each highlighted edge
corresponds to the unique variable $x_{i,j}$ which is finite and non-zero on that component.} 
\label{fig:egvert}
\end{figure}

\begin{proposition} The charts cover $\ol{Q}(n,\br)$, and the transition functions are smooth on the intersections of the charts.  They give $\ol{Q}(n,\br)$ the topology of a smooth manifold-with-corners.  
\end{proposition}
\begin{proof}
It is straightforward to check that every element of $\ol{Q}(n,\br)$ lies in the chart of any maximal combinatorial type refining its combinatorial type.  It is also immediate from the construction that the transition functions are rational functions in the chart coordinates.   Since charts only intersect where the coordinates where they differ are finite and non-zero, the transition functions are smooth. 
\end{proof}

\section{Connection with graph associahedra}

Graph associahedra were introduced by Carr and Devadoss \cite{cardev}, and given convex hull realizations with integer coordinates by Devadoss \cite{dev09}.   We begin this section by recalling the definitions of graph associahedra from \cite{cardev}, as well as the algorithm for integer coordinates of \cite{dev09}.

Let $\Gamma = (V,E)$ be a finite graph with vertices $V$ and edges $E$, with no multiple edges and no loops. 


\begin{definition}
A {\em tube} is a proper connected subset of vertices of $\Gamma$.   A {\em tubing} of $\Gamma$ is a collection of tubes, such that each pair of tubes in the tubing satisfies the following admissibility conditions:
\begin{enumerate}
\item a pair of tubes may be {\em nested}  provided that the inner tube is a proper subset of the outer tube.

\item a pair of tubes may be {\em disjoint} provided that there is no edge connecting a vertex in one tube with a vertex in the other tube.   
\end{enumerate}

\end{definition}

\begin{example}
Figure \ref{fig:tubes} shows examples of non-admissible and admissible tubings for the same underlying graph.  
\end{example}

\begin{figure}[h]
\center{
\includegraphics[height=0.7in]{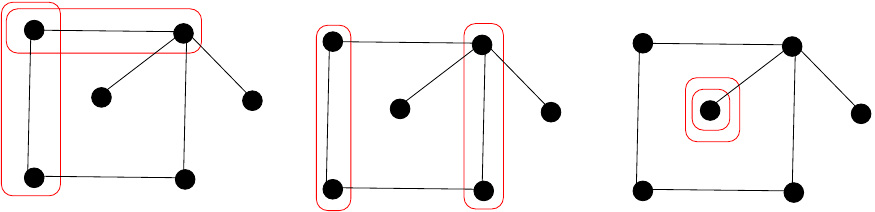} \\
\includegraphics[height=0.7in]{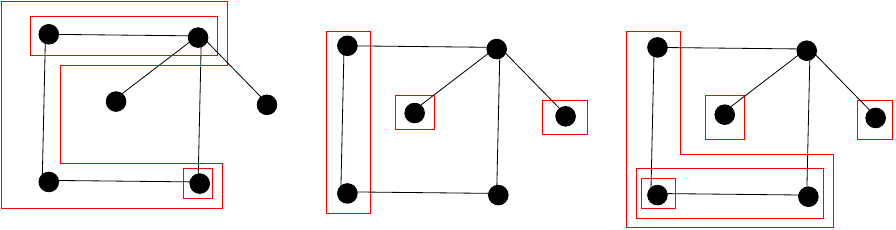}
}
\caption{Top row: non-admissible tubings.  Bottom row: examples of admissible tubings.  The rightmost tubing is a maximal tubing.   }\label{fig:tubes}
\end{figure}

\begin{definition} The {\em graph associahedron} $K_{\Gamma}$ is a simple polytope whose facets of codimension $k$ are indexed by tubings containing $k$ tubes.  A facet of codimension $k$ is contained in a facet of codimension $k^\prime$ if and only if the set of tubes in the $k^\prime$-tubing contains the set of tubes in the $k$-tubing.  
\end{definition}
In particular, maximal tubings correspond to zero dimensional facets. 

\begin{definition} Two maximal tubings are called {\em adjacent} or {\em neighboring} if they differ by precisely one tube.  A tube $T$ is called {\em minimal} for a vertex $v$ of $\Gamma$ if $v$ is in the tube $T$ but is not in any smaller tube.
\end{definition}

It is straightforward to check that:
\begin{lemma}\label{list}
In a maximal tubing of $\Gamma$, each tube $T$ determines a unique vertex $v_T$ of $V$: namely, the vertex for which $T$ is minimal.  When the set of all vertices, $V$, is also counted as a tube, then each maximal tubing determines a bijection between tubes and vertices of $\Gamma$.  
\end{lemma}

Now we recall the algorithm of \cite{dev09} which assigns an integral vector to each maximal tubing, such that the convex hull of all the vectors realizes the corresponding graph associahedron.  

\begin{definition}
To each maximal tubing of $\Gamma$ assign a {\em weight vector}, which is a function $\bw: V \to \Z$, by induction on the number of vertices in a tube:
\begin{enumerate}

\item If $v$ is the only vertex in a tube, $\bw(v)=0$.

\item Otherwise, let $T$ be the minimal tube containing $v$.  By maximality, all other vertices 
in the tube $T$ are contained in tubes of smaller size than $|T|$, so by induction the function $\bw$ is defined on them already.  Then set $\bw(v) = 3^{|T|-2} - \sum_{v^\prime\in T, v^\prime\neq v} \bw(v^\prime)$.  
\end{enumerate}

In other words, if a tube $T$ has two or more vertices in it, the sum of all weights of its vertices should be $3^{|T|-2}$.  For the purposes of this definition, we think of the last vertex in the process (which is not contained in any tube) as belonging to a tube of size $|V|$.  

\end{definition}

\begin{figure}
\center{\includegraphics{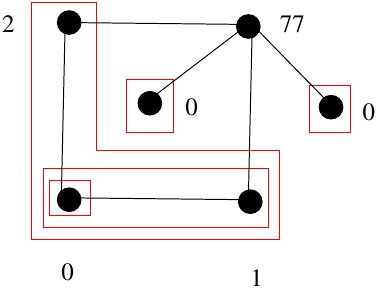}}
\caption{The weights for the maximal tubing of Figure \ref{fig:tubes}.} \label{fig:egwt}
\end{figure}

\subsection{Moduli of quilted lines as graph associahedra}

Now we show how the moduli spaces $\ol{Q}(n,\br)$ are related to certain graph associahedra. 

\begin{definition}  For fixed $n$ and $\br$, the {\em dual graph to $\ol{Q}(n,\br)$} is the graph $\Gamma(n,\br)$ consisting of the complete graph on $n$ vertices $v_1, \ldots, v_n$ with a path of length $r_i - 1$ adjoined to each vertex $v_i$.  
\end{definition}

\begin{remark} \label{rmk}
The graph $\Gamma(n,\br)$ is dual to $Q(n,\br)$ in the following sense:  if one thinks of the elements of $Q(n,\br)$ as rays with markings, then each vertex of the graph $\Gamma(n,\br)$ corresponds to a {\em finite} edge in the rays, and two vertices in $\Gamma(n,\br)$ are connected by an edge if they correspond to {\em adjacent} finite edges in the rays.  In particular, every vertex of $\Gamma(n,\br)$ corresponds to a variable $x_{i,j}$ of Definition \ref{coord}.  A tubing in the graph $\Gamma(n,\br)$ will correspond to all variables $x_{i,j}$ inside the tube going to zero.  
\end{remark}
\begin{example}
See Figure \ref{qgraphfig}.
\end{example}

\begin{figure}[h]
\includegraphics[height=1.5in]{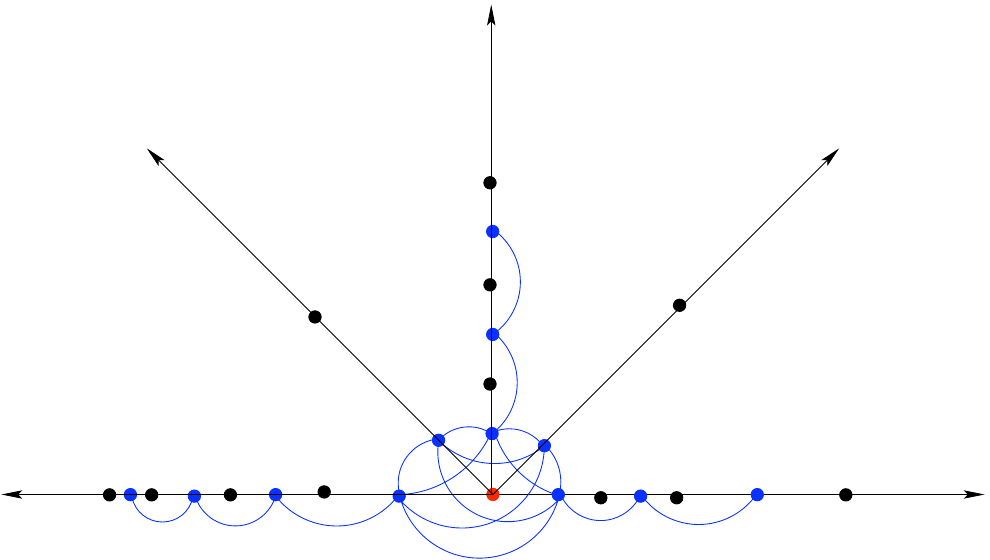} \ \ \ \ \ \ \includegraphics[height=1in]{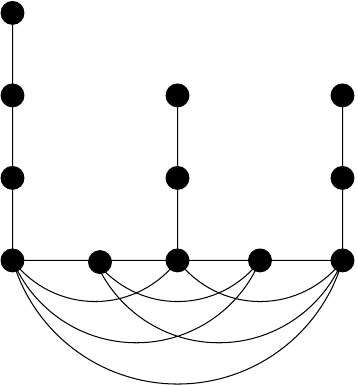}
\caption{{\em Left.} Marked rays in $Q(5,(4,1,3,1,3)$ with the dual graph $\Gamma(5,(4,1,3,1,3))$ superimposed in blue. {\em Right.} $\Gamma(5,(4,1,3,1,3))$ again.}
\label{qgraphfig}
\end{figure} 

\begin{lemma}\label{bij}
There is a canonical bijection between tubings of $\Gamma(n,\br)$ and combinatorial types of $\ol{Q}(n,\br)$ which respects their respective poset structures.  Under this bijection, each tube in a tubing corresponds to a node in the combinatorial type.   For a maximal tubing $\mathcal{T}$, the chart ratio for the node corresponding to a tube $T$ is $\phi = x_{i,j}/x_{k,l}$ where $x_{i,j}$ is the variable labeling the (unique) vertex of $\Gamma(n,\br)$ for which $T$ is maximal, and $x_{k,l}$ is the variable labeling the (unique) vertex of $\Gamma(n,\br)$ for which the smallest tube containing $T$ is maximal. 
\end{lemma}

\begin{proof}
As in Remark \ref{rmk}, the vertices of $\Gamma(n, \br)$ correspond to variables $x_{i,j}$ of Definition \ref{coord}.  Each tube in the tubing determines a node, which separates the marked points that are adjacent to any of the $x_{i,j}$ in the tube, from all the other marked points.  The admissibility conditions on the tubings translate directly into admissibility  conditions on combinatorial types.    

Each tube in a tubing determines a particular disk component in the bubble tree: namely, the component adjacent to the node corresponding to $T$ that is {\em furthest} from the distinguished marking $+\infty$.  Every vertex of $\Gamma(n,\br)$ that is contained in the tube $T$ and which is {\em not} contained in another tube nested in $T$ corresponds to a variable $x_{i,j}$ which is finite and non-zero on that disk component.  
\end{proof}

The next lemma is the main ingredient in the proof of Theorem \ref{thmA}. 

\begin{lemma}\label{tech}
Suppose that two maximal tubings $\mathcal{T}$ and $\widetilde{\mathcal{T}}$ are adjacent, i.e.,  there is a tube $T \in \mathcal{T}$ and a tube $T^\prime \in \widetilde{\mathcal{T}}$ such that the reduced tubings $\mathcal{T} - \{T\} $ and $ \widetilde{\mathcal{T}}-\{T^\prime\}$ are identical.  Let $\bw$ and $\widetilde{\bw}$ be their respective weight vectors.  Let $\phi$ be the chart ratio labeling the node corresponding to the tube $T$ under the bijection of Lemma \ref{bij}.  Then 
\[
\bx^{\widetilde{\bw}-\bw} = \phi^m
\]  
for some $m \geq 1$.  More generally, for any two maximal tubings $\mathcal{T}$ and $\widetilde{\mathcal{T}}$, there are tubes $T_1,\ldots, T_k$ of $\mathcal{T}$ and integers $m_1,\ldots,m_k \geq 1$ such that
\[
\bx^{\widetilde{\bw} - \bw} = \phi_1^{m_1}\phi_2^{m_2} \ldots \phi_k^{m_k}
\]
where $\phi_i$ is the ratio for the node corresponding to the tube $T_i$. 
\end{lemma}
\begin{proof}

The proof is a calculation.  Begin by fixing a labeling of the tubes in the reference tubing $\mathcal{T}$.  Label them $T_1, T_2,\ldots, T_k$, where $k = |V|-1$.   Each tube $T_i$ determines a pair of vertices, the vertex immediately inside $T_i$ and the vertex immediately outside $T_i$. By Lemma \ref{bij} each  vertex is identified with a variable $x_{i,j}$, so we write $x_{in}(T_i)$ and $x_{out}(T_i)$ for the variables corresponding to the vertices immediately inside and immediately outside of $T_i$ respectively.  Note also by Lemma \ref{bij} that the chart ratio for the node corresponding to the tube $T_i$ is precisely $x_{in}(T_i)/x_{out}(T_i)$.  There is exactly one vertex which is not inside a tube, let us label this outermost vertex by $x_*$.   With this notation, it follows from Lemma \ref{list} that a complete indexing of the set of vertices is 
\begin{equation}\label{ind}
V = \{x_{in}(T_1), \ldots, x_{in}(T_k), x_*\}.  
\end{equation}
Let $\bw: V \to \Z$ be the weight vector for $\mathcal{T}$, and label its entries by $w_1, \ldots, w_k, w_*$.  Write $w(T_i)$ for the sum of the weights in the tube $T_i$, and write $w(V) = 3^{|V|-2}$ for the sum of all weights in the graph.  Using this notation we can then write
\bea
\bx^{\bw} & := & x_{in}(T_1)^{w_1} x_{in}(T_2)^{w_2}\ldots x_{in}(T_k)^{w_k} x_*^{w_*} \\
& = & \left(\frac{ x_{in}(T_1)}{x_{out}(T_1)}\right)^{w(T_i)} \left(\frac{ x_{in}(T_2)}{x_{out}(T_2)}\right)^{w(T_2)} \ldots \left(\frac{ x_{in}(T_k)}{x_{out}(T_k)}\right)^{w(T_k)} \ x_*^{w(V)}\\
& =: & \phi_1^{w(T_1)} \ldots \phi_k^{w(T_k)} x_*^{w(V)}.
\eea

Now consider the weights for another maximal tubing $\widetilde{\mathcal{T}}$, with weight vector $\tilde{w}^\prime: V \to \Z$, with entries $\widetilde{w}_1, \ldots, \widetilde{w}_k, \widetilde{w}_*$ with respect to the indexing \eqref{ind}, and writing $\widetilde{w}(T_i)$ for the sum of the weights in the tube $T_i$.  We have that
\bea
\bx^{\widetilde{\bw}} & := & x_{in}(T_1)^{\widetilde{w}_1} x_{in}(T_2)^{\widetilde{w}_2} \ldots x_{in}(T_k)^{\widetilde{w}_k} x_*^{\widetilde{w}_*}\\
& = & \left(\frac{ x_{in}(T_1)}{x_{out}(T_1)}\right)^{\widetilde{w}(T_i)} \left(\frac{ x_{in}(T_2)}{x_{out}(T_2)}\right)^{\widetilde{w}(T_2)} \ldots \left(\frac{ x_{in}(T_k)}{x_{out}(T_k)}\right)^{\widetilde{w}(T_k)} \ x_*^{\widetilde{w}(V)}\\
& = : &   \phi_1^{\widetilde{w}(T_1)} \ldots \phi_k^{\widetilde{w}(T_k)} x_*^{\widetilde{w}(V)}.
\eea

Since $w(V) = \widetilde{w}(V)$, it follows that 
\bea
\bx^{\widetilde{\bw} - \bw} & =  \phi_1^{\widetilde{w}(T_1) - w(T_1)} \ldots \phi_k^{\widetilde{w}(T_k) - w(T_k)},
\eea
and we claim that $\widetilde{w}(T_i) - w(T_i) \geq 0$ for all $i = 1,\ldots, k$.  To prove the claim, suppose first that $T_i$ is a tube in the tubing $\widetilde{\mathcal{T}}$ as well.  Then, $\widetilde{w}(T_i) - w(T_i) = 0$, since the sum of weights in a tube is completely determined by the number of vertices in the tube.  Now suppose that $T_i$ is {\em not} a tube in $\widetilde{\mathcal{T}}$.   But in this case, there is at least one vertex in $T_i$ that, in the tubing $\widetilde{\mathcal{T}}$, is the maximal vertex for a tube which is strictly larger than $T_i$.     The algorithm for the weight vectors has the property that the maximal vertex in a tube always has a weight strictly greater than the total weight of any smaller tube.  Let $v \in T_i$ be such a vertex, which with respect to the tubing $\widetilde{\mathcal{T}}$ is maximal for a tube that strictly contains $T_i$.  Then, 
\[
\widetilde{w}(T_i)  \geq \widetilde{w}(v) > w(T_i).
\]
This proves the claim, and the general statement of the Lemma.

The first part of the Lemma is a special case: without loss of generality, suppose that the tube $T_1$ is the only tube in $\mathcal{T}$ that is different from the tubing in $\widetilde{\mathcal{T}}$.  Then we have
\bea
\bx^{\widetilde{\bw} - \bw} = \phi_1^{m_1},
\eea 
where $m_1 \geq 1$ since $T_1$ is {\em not} a tube in $\widetilde{\mathcal{T}}$.
\end{proof}

Now we can prove Theorem \ref{thmA}, namely:
\begin{thmA} 
$\ol{Q}(n,\br)$ is homeomorphic to the graph associahedron for  $\Gamma(n,\br)$. 
\end{thmA}

\begin{proof} 
To prove the homeomorphism, we identify $\ol{Q}(n,\br)$ with the non-negative real part of a projective toric variety defined by the weight vectors of the convex hull realization of \cite{dev09}.  By the general theory of toric varieties (see e.g. \cite{sottile-2002}, \cite{fulton-intro}), the moment map is a homeomorphism from the non-negative real part of the toric variety to the convex hull of the weight vectors.  

Let $\bx = (x_{1,1}: \ldots : x_{1,r_1}: \ldots : x_{n,1}: \ldots : x_{n,r_n})$ be the homogeneous coordinates  on $Q(n,\br)$ of Definition \ref{coord}, and write $\bz$ for the complexification of these  homogeneous coordinates.  Let $\mathcal{T}_1, \ldots, \mathcal{T}_N$ be all the maximal tubings  of $\Gamma(n,\br)$,  with weights $\bw_1, \ldots, \bw_N$.    
Let $X \subset \C P^{N-1}$ be the projective toric variety defined by taking the closure in $\C P^{N-1}$ of the image of the map 
\begin{eqnarray}
\bz \mapsto (\bz^{\bw_1} : \ldots : \bz^{\bw_N}).  \label{embed}
\end{eqnarray} 
The map \eqref{embed} defines an embedding $Q(n,\br)  \hookrightarrow X$ sending $\bx   \mapsto  (\bx^{\bw_1}: \bx^{\bw_2} : \ldots : \bx^{\bw_N}),$
whose image is a dense subset of the non-negative real part of $X$.  We will show that the embedding extends to the compactification $\ol{Q}(n,\br)$, and that the non-negative part of $X$ is homeomorphic to $\ol{Q}(n,\br)$.  

Write $\mathbb{A}_1, \ldots, \mathbb{A}_N$ for the affine charts on $\C P^{N-1}$.  That is, $\mathbb{A}_i$ consists of all elements of $\C P^{N-1}$ parametrized so that the $i$-th coordinate is 1. 

Consider the affine slice $X_i:= X \cap \mathbb{A}_i$.  We will prove that the non-negative part of $X_i$ is homeomorphic to the ratio chart associated to the $0$-dimensional facet of $\ol{Q}(n,\br)$ whose combinatorial type corresponds to the maximal tubing $\mathcal{T}_i$.      
 In the affine slice $X_i$, the image of $Q(n,\br)$ under the map \eqref{embed} is 
\bea
(\bx^{\bw_1 - \bw_i} : \ldots : \underset{i}{\underbrace{1}} : \ldots : \bx^{\bw_N - \bw_i}).
\eea 
By Lemma \ref{tech} we get a smooth identification 
\[
(\phi_1, \ldots, \phi_k): [0,\infty)^{k} \to  (\bx^{\bw_1 - \bw_i} : \ldots : \underset{i}{\underbrace{1}} : \ldots : \bx^{\bw_N - \bw_i})
\] 
of the ratio chart for the $i$-th tubing with $\ol{Q}(n,\br) \cap X_i$, since on the domain $[0,\infty)$ the map $\phi_k \mapsto \phi_k^{m}$ for $m\geq 1$ is a smooth homeomorphism.    The non-negative part of the toric variety is therefore homeomorphic to $\ol{Q}(n,\br)$, since the non-negative part of each affine chart is identified with a ratio chart of $\ol{Q}(n,\br)$.  The non-negative part of the toric variety (hence $\ol{Q}(n,\br)$) is homeomorphic by the moment map to the convex hull of the weight vectors (see e.g. \cite{sottile-2002}), which in turn realize the graph associahedron, by \cite{dev09}.  
\end{proof}

\subsection{Special cases: associahedra and permutahedra}

For $\br \in \Z^n_{\geq 0}$, let $|\supp \br|$ be the number of non-zero entries in $\br$.  By stability,  $|\supp \br| \geq 1$.

It is shown in \cite{cardev} that in the special case that a graph is the complete graph on $n$ vertices, the corresponding graph associahedron is the $n$-th permutahedron, and in the special case that a graph is a path on $n$ vertices, the corresponding graph associahedron is the Stasheff associahedron $K_n$. So an immediate corollary of Theorem \ref{thmA} is:

\begin{corollary} \begin{enumerate}
\item If $\br$ contains only 1's and 0's, then $\ol{Q}(n,\br)$ is the $|\supp \br |$-th permutahedron.  
\item If $|\supp \br| \leq 2$,  then $\ol{Q}(n,\br)$ is the $(1 + \sum\limits_{i=1}^n r_i)$-th associahedron. 
\end{enumerate}
\end{corollary}

\begin{example} Figure \ref{2deg} shows the simplest two dimensional examples.
\begin{figure}[h]
\includegraphics[height=1.5in]{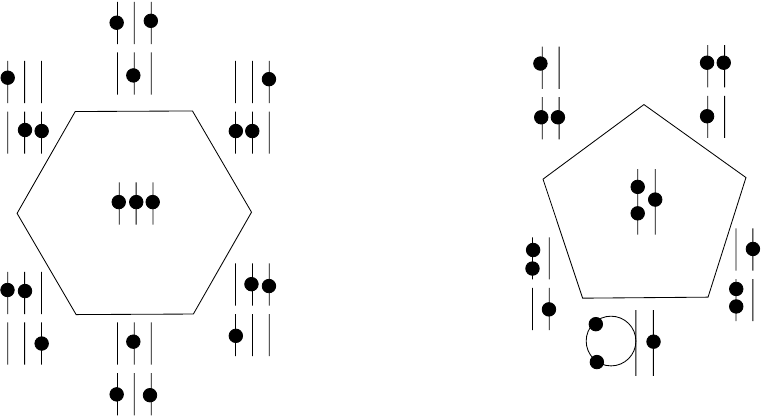}
\caption{$\ol{Q}(3,(1,1,1)$ realizes the permutahedron $P_3$, while $\ol{Q}(2,(2,1))$ realizes the associahedron $K_4$. }\label{2deg}
\end{figure}
\end{example}

\section{\ainfty  $n$-modules} \label{ainfty}

We end by illustrating how the combinatorics of moduli spaces of quilted strips govern the algebraic structure of \ainfty  $n$-modules. Roughly speaking an \ainfty $n$-module is a vector space equipped with multiplication operations by \ainfty algebras $\A_1,\ldots,\A_n$, such that multiplication by elements of the same $\A_i$ is homotopy associative, and multiplication by elements of different \ainfty algebras $\A_i$ and $\A_j$ is homotopy commutative.  

General references for \ainfty modules and bimodules are \cite{lefevre},  \cite{fooo}, \cite{seidel-ainfty}.   We give a simplified definition which does not take into account gradings or signs, but suffices to illustrate the role of the codimension one facets of the graph associahedra.  

 \begin{definition} Let $(\A_1, \ldots, \A_n)$ be an $n$-tuple of \ainfty algebras.   An \ainfty {\em $n$-module} for $(\A_1, \ldots, \A_n)$ is a $\Z_2$-vector space $M$, together with a collection of multilinear maps
\[
\mu^{\br |1} : \A_1^{\otimes r_1}\otimes \ldots \otimes \A_n^{\otimes r_n}\otimes M  \to  M
\]
for $\br = (r_1, \ldots, r_n)  \in \Z_{\geq 0}^n$.  These multilinear maps satisfy the {\em \ainfty $n$-module equations}, which for inputs $(\bfa_1;\ldots;\bfa_n;m)\in  \A_1^{\otimes r_1}\otimes \ldots \otimes \A_n^{\otimes r_n}\otimes M$ look like
\begin{eqnarray}\label{nmod}
0 &= & \sum\limits_{i,j, k} \mu^{\br - (k-1)\mathbf{e}_i |1}(\bfa_1;\ldots; \bfa_i^{1:j}, \mu^k_{\A_j}(\bfa^{j+1:j+k}), \bfa^{j+k+1: r_i}_{i}; \ldots ; \bfa_n; m)\nonumber \\
 && + \sum\limits_{0\leq \mathbf{s} \leq \br} \mu^{\br - \mathbf{s} |1}(\bfa^{s_1+1:r_1}_1; \ldots; \bfa^{s_n +1: r_n}_n; \mu^{\mathbf{s} |1}(\bfa_1^{1:s_1}; \ldots; \bfa^{1:s_n}_n; m)).
\end{eqnarray}
We use bold notation for inputs $\bfa_i = (a_i^1,\ldots,a_i^{r_i}) \in \A_i^{\otimes r_i}$, and for $x\leq y$ we use the shorthand $\bfa_i^{x:y}:=(a_i^{x}, a_i^{x+1},\ldots, a_i^y)$.   We also use the notation $0\leq\bfs \leq \br$ to mean that for each $i$, $0\leq s_i \leq r_i$.  
\end{definition}

The definition for non-unital \ainfty categories $(\A_1, \ldots, \A_n)$ is analogous:  for each tuple $(L_1, \ldots, L_n) \in \Ob \A_1 \times \ldots \times \Ob \A_n$ there is an associated $\Z_2$-vector space $M(L_1,\ldots, L_n)$, and given objects $(L_1^0, \ldots, L_1^{r_1}, \ldots, L_n^0, \ldots, L_n^{r_n})$ in $(\Ob \A_1 )^{r_1+1} \times \ldots \times (\Ob \A_n)^{r_n+1}$ there are multilinear maps
\bea
 \mu^{\br | 1} : & &  \left( \bigotimes_{i=1}^{r_1} \hom(L_1^{i-1}, L_1^i)\right) \otimes   \left( \bigotimes_{i=1}^{r_2} \hom(L_2^{i-1}, L_2^i)\right) \otimes \ldots \\
& & \ldots \left( \bigotimes_{i=1}^{r_n} \hom(L_n^{i-1}, L_n^i) \right) \otimes M(L_1^0, \ldots, L_n^0) \to M(L_1^{r_1}, \ldots, L_n^{r_n}) 
\eea
which for fixed inputs satisfy the \ainfty $n$-module equations \eqref{nmod}.

The connection with the moduli spaces comes from visualizing the maps $\mu^{\br | 1}$ as marked $n$-quilted lines with inputs at the markings on the lines and at the distinguished end $+\infty$, and the output labeling the distinguished end $-\infty$.  The $n$ lines are labeled by $\A_1, \ldots, \A_n$ respectively, with marked points along the $i$-th line labeled by elements of the algebra $\A_i$, and the $+\infty$ end of the quilted line is labeled by an element of the $n$-module $M$.  

\begin{figure}[h]
\center{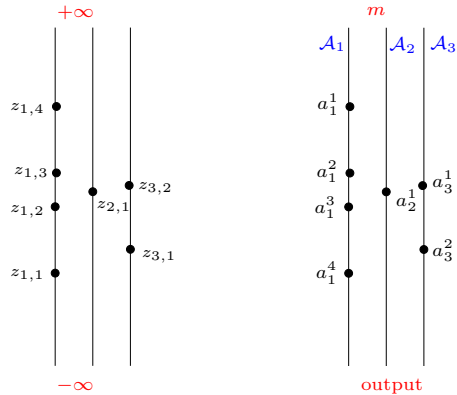
}
\caption{A quilted line in $Q(3, (4,1,2))$ representing the operation $\mu^{4,1,2|1}$ applied to inputs from $\A_1, \A_2, \A_3$ and $M$. }
\end{figure}

The terms appearing on the right hand side of \eqref{nmod} are of two types, stable or unstable.  The unstable terms are those which contain either $\mu^{{\mathbf{0}} | 1}$ or a $\mu_{\A_j}^1$.   All other terms are stable, indexed by the codimension one facets of $\ol{Q}(n,\br)$, in which either a number of markings on the same line can bubble off, or the quilted line can break (see Figure \ref{codim1}).  

\begin{figure}[h]
\center{
\includegraphics[height=1in]{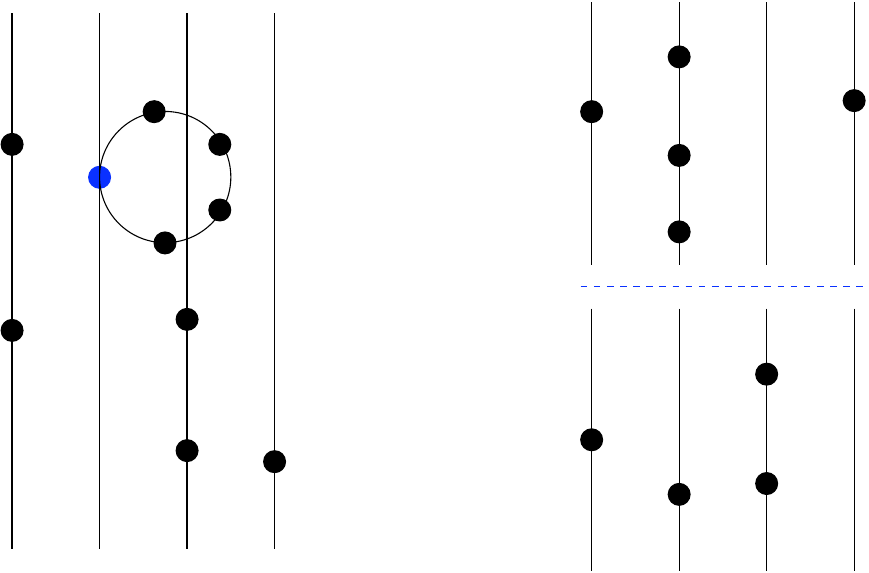}
}
\caption{The two types of codimension one facet.}\label{codim1}
\end{figure}

\subsection{Morphisms}

Let $(\A_1, \ldots, \A_n)$ be \ainfty algebras.  Let $M$ and $N$ be $(\A_1, \ldots, \A_n)$-modules.   A morphism $\theta : M \to N$ is a collection of maps, 
\bea
\theta^{\br | 1}: \A_1^{\otimes r_1} \otimes \ldots \otimes \A_n^{\otimes r_n} \otimes M \to N,
\eea
for $\br = (r_1, \ldots, r_n)$, with $r_i \geq 0$.  The composition of morphisms $\psi: P \to M, \theta: M\to N$,  is defined by summing over all ways of applying $\psi$ followed by $\theta$ to the inputs, and is associative on the nose (just like the composition of morphisms of \ainfty modules or bimodules is associative, see e.g. \cite{seidel-ainfty}).   There is also a boundary operator $\partial$ on morphisms, 
\bea
& & (\partial \theta )^{\br |1}(\bfa_1;\ldots;\bfa_n; m)  := \\
 & & \sum_{i,j,k} \theta^{\br - k\bfe_i | 1} (\bfa_1; \ldots; \bfa_i^{1:j},\mu_{\A_i}^k(\bfa_i^{j+1: j+k}), \bfa_i^{j+k+1: r_i}; \ldots; \bfa_n; m)\\
 & &+ \sum_{\bfs \leq \br}  \theta^{\br - \bfs | 1} (\bfa_1^{s_1+1:r_1}; \ldots; \bfa_n^{s_n+1: r_n}; \mu_{M}^{\bfs| 1}(\bfa_1^{1:s_1}; \ldots; \bfa_n^{1:s_n}; m))\\
& & +\sum_{\bfs \leq \br}\mu_N^{\br-\bfs|1}( \bfa_1^{s_1+1:r_1}; \ldots; \bfa_n^{s_n+1: r_n}; \theta^{\bfs| 1}(\bfa_1^{1:s_1}; \ldots; \bfa_n^{1:s_n}; m)),
\eea
generalizing the boundary operator for morphisms of \ainfty modules or bimodules \cite{seidel-ainfty}.  Write $n-mod(\A_1, \ldots, \A_n)$ for the dg-category whose objects are $(\A_1, \ldots, \A_n)$-modules.  
For \ainfty categories $\A_1,\ldots,\A_n$ the definitions are nearly identical, the only difference being the additional indexing by ordered tuples of objects and their ordered $\hom$ spaces. 
\subsection{Induced functors}
When $\A$ and $\B$ are \ainfty categories, it is well-known that an $(\A,\B)$-bimodule determines an \ainfty functor from $\A$ to $\mbox{mod}\B$.  A completely analogous property holds for $n$-modules.  
We write $(x_1, \ldots, \widehat{x}_i, \ldots, x_{k})$ to denote the $(k-1)$-tuple obtained from $(x_1, \ldots, x_k)$ by deleting the $i$-th term.  
\begin{proposition}
Let $M(\A_1, \ldots, \A_{n+1})$ be an $(n+1)$-module, for some $n\geq 1$ and \ainfty categories $\A_1,\ldots,\A_n$.   Then for each $i = 1, \ldots, n+1$ there is an induced \ainfty functor \[
\F_i : \A_i \to n\mbox{-mod}(\A_1, \ldots, \widehat{\A_i}, \ldots, \A_{n+1}).
\]
\end{proposition}
 
\begin{proof}
Explicitly, on objects,  $\F_i(A)$ for $A \in \A_i$ is the $n$-module obtained from $M$ by inserting the object $A$ in the $i$-th slot.   On morphisms, the functor consists of multilinear maps
\bea
\F_i^k : \hom_{\A_i}(A_0,A_1) \otimes \ldots \otimes \hom_{\A_i}(A_{k-1}, A_k) \to \hom_{n-mod}(\F_i(A_0), \F_i(A_k)),
\eea
where $\mathbf{a} = (a_1, \ldots, a_k)$ is mapped to the morphism of $n$-modules given by the multilinear maps \bea
\theta^{r_1, \ldots, \widehat{r}_i, \ldots, r_{n+1}  | 1} (\ldots) := \mu^{r_1,\ldots,  k, \ldots, r_n | 1}(\ldots ; \bfa; \ldots),   
 \eea
 where the $\mu^{\br|1}$ are the $(n+1)$-module maps of $M$.  It is straightforward to check that the \ainfty $(n+1)$-module equations for $M$ translate into the \ainfty functor relations for $\F_i$.  
\end{proof}

\bibliographystyle{plain}
\def\cprime{$'$}

\end{document}